\documentclass[11pt]{amsart}
\usepackage{amssymb,amsmath,amsfonts,amscd,euscript}

\newcommand{\nc}{\newcommand}

\numberwithin{equation}{section}
\newtheorem{thm}{Theorem}[section]
\newtheorem{prop}[thm]{Proposition}
\newtheorem{lem}[thm]{Lemma}
\newtheorem{cor}[thm]{Corollary}
\theoremstyle{remark}
\newtheorem{rem}[thm]{Remark}
\newtheorem{nota}[thm]{Notation}

\newtheorem{conj}[thm]{Conjecture}

\nc{\gl}{\mathfrak{gl}}
\nc{\GL}{\mathfrak{GL}}
\nc{\g}{\mathfrak{g}}
\nc{\frk}{\mathfrak{k}}
\nc{\gh}{\widehat\g}
\nc{\h}{\mathfrak{h}}
\nc{\la}{\lambda}
\nc{\C}{\mathbb C }
\nc{\Z}{\mathbb Z }
\nc{\N}{\mathbb N }
\nc{\R}{\mathbb R }
\nc{\Q}{\mathbb Q }
\nc{\VV}{\mathbb V }
\nc{\VVV}{{\mathbb V\hskip -7pt \VV}} %\hbox{\rm V}}}
\nc{\Y}{\widehat Y}
\nc{\al}{\alpha }
\nc{\om}{\omega}

\nc{\ta}{\theta}
\nc{\ve}{\varepsilon}
\nc{\ch}{{\mathop {\rm ch}}}
\nc{\Tr}{{\mathop {\rm Tr}\,}}
\nc{\Id}{{\mathop {\rm Id}}}
\nc{\ad}{{\mathop {\rm ad}}}
\nc{\eins}{{1\hskip -3.5pt \hbox{I}}}
\nc{\bra}{\langle}
\nc{\ket}{\rangle}
\nc{\x}{{\bf x}}
\nc{\pa}{\partial}
\nc{\ld}{\ldots}
\nc{\cd}{\cdots}
\nc{\chara}{{\mathop {\rm char}\,}}
\nc{\sign}{{\mathop {\rm sign}}}
\nc{\Lie}{{\mathop {\rm Lie}}}
\nc{\res}{{\mathop {\rm res}}}
\nc{\trdg}{{\mathop {\rm trdeg}\,}}
\nc{\diag}{{\mathop {\rm diag}}}
\nc{\hk}{\hookrightarrow}
\nc{\T}{\otimes}
\newcommand{\bea}{\begin{equation}}
\newcommand{\ena}{\end{equation}}
\newcommand{\be}{\begin{equation*}}
\newcommand{\en}{\end{equation*}}
\nc{\gr}{\mathrm{gr}}
\nc{\ov}{\overline}

\nc{\cO}{\mathcal O}
\nc{\msl}{\mathfrak{sl}}
\nc{\msp}{\mathfrak{sp}}
\nc{\mso}{\mathfrak{so}}
\nc{\mgl}{\mathfrak{gl}}
\nc{\U}{\mathrm U}
\nc{\V}{\EuScript V}

\newcommand{\bc}{{\mathbb C}}

\newcommand{\bn}{{\mathbb N}}
\newcommand{\fp}{{\mathfrak p}}
\newcommand{\fg}{{\mathfrak g}}

\newcommand{\ft}{{\mathfrak t}}

\newcommand{\fn}{{\mathfrak n}}
\newcommand{\fl}{{\mathfrak l}}

\begin{document}

\title[Zhu's algebra and the $C_2$-algebra]
{Zhu's algebra and the $C_2$-algebra in the symplectic and the orthogonal cases}

\author{Evgeny Feigin and Peter Littelmann}
\address{Evgeny Feigin:\newline
Tamm Theory Division,
Lebedev Physics Institute,\newline
Leninisky prospect, 53,
119991, Moscow, Russia,\newline
{\it and }\newline
French-Russian Poncelet Laboratory, Independent University of Moscow
}
\email{evgfeig@gmail.com}
\address{Peter Littelmann:\newline
Mathematisches Institut, Universit\"at zu K\"oln,\newline
Weyertal 86-90, D-50931 K\"oln,Germany
}
\email{littelma@math.uni-koeln.de}

\begin{abstract}
We prove that Zhu's algebra and the $C_2$-algebra of type ${\tt C}_m$
have the same dimension, and we compute the graded character
of the latter. Maximal parabolic subalgebras of the symplectic algebra
play a central role in our construction. For the orthogonal algebras
our methods do not allow to describe the whole $C_2$-algebras, 
we get only a description of a certain quotient of the algebra.
\end{abstract}

\maketitle

\section*{Introduction}
In this paper we continue the study of Zhu's algebras and $C_2$-algebras of WZW models,
initiated in \cite{FFL}. We briefly recall the setup.

The theory of vertex operator algebras plays a key role in the mathematical
description of the structures arising in the conformal field theories
(see \cite{GabGod}, \cite{BF},\cite{K2}). In particular spaces of states, partition functions
and amplitudes can be described via the structure theory of vertex operator
algebras and their representation theory.

The representation theory of a vertex operator algebra $\V$ is in general very complicated.
But in some special cases (so-called rational VOAs) it is controlled by
a certain finite-dimensional semisimple associative algebra $A(\V)$, called Zhu's algebra
(see \cite{Z}, \cite{FZ}).
More precisely, irreducible representations of $A(\V)$ are in one-to-one correspondence
with irreducible representations of $\V$.
Zhu's algebra can be explicitly computed in some special cases (for example for
minimal models and WZW models). In this paper we only deal with the WZW
models associated with a simple Lie algebra $\g$ on the non-negative integer
level $k$. The corresponding Zhu's algebra (denoted by $A(\g;k)$)
is given by (see \cite{FZ})
$$
A(\g;k)=\U(\g)/\langle e_\theta^{k+1}\rangle,
$$
where $\theta$ denotes the highest root, $e_\theta\in\g$ is a corresponding root vector,
and $\langle e_\theta^{k+1}\rangle$ is the two sided ideal in the universal enveloping
algebra of $\g$ generated by $e_\theta^{k+1}$.

As mentioned above, the rationality is a very important characterization of a VOA.
There exists a condition (called the $C_2$-cofiniteness condition,
see \cite{Z}, \cite{CF},\cite{M1},\cite{M2}) which guaranties
the rationality. Namely,  the $C_2$-algebra $A_{[2]}(\V)$
attached to $\V$ is defined as a quotient of $\V$ by the linear span of the elements
of the form $a_{n}b$, $n\le -2$, with $a,b\in \V$ and $a_n$ being Fourier
modes of the field corresponding to $a$.  The $C_2$-cofiniteness condition reads as
$\dim A_{[2]}(\V)<\infty$. 
The algebras $A(\V)$ and $A_{[2]}(\V)$ are very closely related (see \cite{GabGod},
\cite{GG},\cite{GN},\cite{N}).
For WZW models one has
$$
A_{[2]}(\g;k)=S^\bullet(\g)/\langle \U(\g)\circ e_\theta^{k+1}\rangle,
$$
where $S^\bullet(\g)$ denotes the symmetric algebra of $\g$,
``$\,\circ\,$''denotes the action of $\g$ on $S^\bullet(\g)$ induced by the adjoint action,
$\U(\g)\circ e_\theta^{k+1}\subset S^\bullet(\g)$ is the (irreducible) $\g$-module
generated by $e_\theta^{k+1}\in S^{k+1}(\g)$, and $\langle \U(\g)\circ e_\theta^{k+1}\rangle$
is the ideal generated by this subspace.

It turned out that in several cases Zhu's algebra and the $C_2$-algebra have the 
same dimension
(see \cite{GG}, \cite{FFL}).
We note, however, that in general $A_{[2]}(\V)$ can be viewed as a
degeneration of $A(\V)$ and thus may be bigger (this happens for example
in the case of the WZW model of type $E_8$ with $k=1$). We note also
that the $C_2$-algebra has an extra grading missing in Zhu's algebra.
For example, for $A_{[2]}(\g;k)$ this grading is inherited from the
degree grading on $S^\bullet(\g)$. It is therefore natural
to ask for the graded dimension (or graded character) of the $C_2$-algebra.

In \cite{GG} Gaberdiel and Gannon conjectured an explicit formula
for the graded character of $A_{[2]}(\g;k)$ for $\g=\msl_n$.
They also conjectured that for all $n$ and $k$ one has
$\dim A(\msl_n;k)=\dim A_{[2]}(\msl_n;k)$. These conjectures were proved in
\cite{FFL}.
In general, we have the following conjecture:
\begin{conj}\label{conj}
For all classical Lie algebras and all $k\ge 0$ one has
$$
\dim A(\g;k)=\dim A_{[2]}(\g;k).
$$
\end{conj}
In this paper we prove the conjecture for $\g$ of type
${\tt C}$ ($\g=\msp_{2m}$) and we compute the graded character of
$A_{[2]}(\msp_{2m};k)$.   

Our main idea is to identify the $C_2$-algebra
$A_{[2]}(\msp_{2m};k)$  with the irreducible representation
$\VVV(k\omega_{2m})$ of the bigger algebra $\msp_{4m}$. We
then use the geometry of the affine cone over the flag variety
$Sp_{4m}/P_{2m}$ and the restriction formulas by Littlewood, Koike
and Terada \cite{Lw,KT} 
to prove the graded character formula. 
Here $P_{2m}\subset Sp_{4m}$ is the parabolic subgroup associated to
the fundamental weight $\omega_{2m}$.

Unfortunately, our methods do not generalize to the
orthogonal Lie algebras. The $C_2$-algebra $A_{[2]}(\mso_n;k)$ turns out
to be bigger than the corresponding representation of $\mso_{2n}$ and
therefore we can only describe a quotient of the $C_2$-algebra.
We still show that this quotient (as a $\g$-module) can be very naturally 
identified with a certain subspace of the Zhu's algebra.

The paper is organized as follows:\\
In Section $1$ we recall the definitions of Zhu's algebra and the $C_2$-algebra,
and we formulate the problem.\\
In Sections $2$, $3$ and $4$ we work out the case of the symplectic
algebra $\msp_{2m}$. In Section $2$ we describe the connection between
the $C_2$ -algebra $A_{[2]}(\msp_{2m};k)$ and certain representations
of $\msp_{4m}$.\\
In Section $3$ we compute the graded character of
$A_{[2]}(\msp_{2m};k)$.\\
In Section $4$ we prove Conjecture \ref{conj}
in type $C_m$.\\
Finally, in Section $5$ we discuss the case of the orthogonal algebras.

\section{Zhu's algebra and the $C_2$-algebra}\label{setup}
\subsection{Definitions}
Let $\g$ be a simple Lie algebra. Let $\theta$ be the highest root of $\g$
and let $e_\theta\in\g$ be a highest weight
vector in the adjoint representation.
Fix a non-negative integer $k$.
Let $P_k^+(\g)$ be the set of level $k$ integrable $\g$ weights, i.e
the set of dominant integral $\g$ weights $\la$ satisfying $(\la,\theta)\le k$
(the Killing form is normalized by the usual condition $(\theta,\theta)=2$).
The following Theorem 1.1 is proved in  \cite{FZ}:

\begin{thm}\label{zhualgebra}
The level $k$ Zhu algebra $A(\g;k)$ is the quotient
of the universal enveloping algebra $\U(\g)$ by the two-sided
ideal generated by $e_\theta^{k+1}$:
\be
A(\g;k)=\U(\g)/\langle e_\theta^{k+1}\rangle.
\en
In addition, one has an isomorphism of $\g$-modules:
\[
A(\g;k)\simeq\bigoplus_{\la\in P_k^+(\g)} V(\la)\T V(\la)^*.
\]
\end{thm}
The form of the description of $A(\g;k)$ arises ultimately because of the
Peter-Weyl Theorem.
\begin{nota}
Let $S(\g)=\bigoplus_{m=0}^\infty S^m(\g)$
be the symmetric algebra of $\g$.
For $v\in S^m(\g)$ and $a\in\g$ let $av\in S^{m+1}(\g)$
be the product in the symmetric algebra.
Each homogeneous summand $S^m(\g)$ is
a $\g$-module by the adjoint action on $\g$. For $v\in S^m(\g)$ and $a\in \g$ we denote by $a\circ v\in S^m(\g)$
the adjoint action of $a$.
\end{nota}

The $C_2$-algebra associated with $\V(\g;k)$ can be described as follows.
The level $k$ $C_2$-algebra $A_{[2]}(\g;k)$ is the quotient
of the symmetric algebra $S(\g)$ by the
ideal generated by the subspace $V_{k+1}=\U(\g)\circ e_\theta^{k+1}\hk S^{k+1}(\g)$:
\be
A_{[2]}(\g;k)=S(\g)/\langle V_{k+1}\rangle.
\en
\begin{rem}
The subspace $V_{k+1}\hk S^{k+1}(\g)$ is isomorphic to the
irreducible $\g$-module $V((k+1)\theta)$ of highest weight $(k+1)\theta$.
The algebra $A_{[2]}(\g;k)$ is naturally a $\g$-module,
the module structure being induced by the adjoint action. Note that
$A_{[2]}(\g;k)$ is {\bf not} a $\g\oplus\g$-module, differently from $A(\g;k)$.
\end{rem}

Consider the standard filtration $F_\bullet$ on the universal enveloping
algebra $\U(\g)$, such that $\gr_\bullet F\simeq S(\g)$. Let
$F_\bullet(k)$ be the induced filtration on the quotient algebra
$A(\g;k)$.
We have an obvious surjection
\begin{equation}
A_{[2]}(\g;k)\to \gr_\bullet F(k).
\end{equation}
Therefore, we have a surjective homomorphism of $\g$-modules
\begin{equation}\label{ge}
A_{[2]}(\g;k)\to A(\g;k)
\end{equation}
and thus
$\dim A_{[2]}(\g;k)\ge \sum_{\beta\in P_k^+(\g)} (\dim V_\beta)^2 $.
A natural question is: when does this inequality turn into an equality?
In this paper we are also interested in the degree grading on $A_{[2]}(\g;k)$ and in the
corresponding graded decomposition into the direct sum of $\g$-modules. Let
\[
S(\g)=\bigoplus_{m\ge 0} S^m(\g)
\]
be the degree decomposition of the symmetric algebra. This decomposition induces
the decomposition of the $C_2$-algebra:
\begin{equation}
A_{[2]}(\g;k)=\bigoplus_{m\ge 0} A_{[2]}^m(\g;k).
\end{equation}
Each space $A_{[2]}^m(\g;k)$ is naturally a representation of $\g$.
Our main questions are as follows:
\begin{enumerate}
\item \label{dimeq}
Prove the equality of the dimensions of $A_{[2]}(\g;k)$ and $A(\g;k)$.
\item \label{grch}
Find the decomposition of $A_{[2]}^m(\g;k)$
into the direct sum of irreducible $\g$-modules.
\end{enumerate}
Recall that the case $\g=\msl_n$ was considered in \cite{GG} and \cite{FFL}.
To be more precise, it was conjectured in \cite{GG} and proved in \cite{FFL} that
\[
\dim A_{[2]}(\msl_n;k)=\dim A(\msl_n;k)
\]
and, as $\msl_n$-modules (not as $\msl_n\oplus\msl_n$-modules, despite the description
as tensor products), one has a decomposition
\[
A^m_{[2]}(k)=\frac
{\bigoplus_{\substack{\la:\ k\ge \la_1,\ \la_n\ge 0\\ \la_1+\dots + \la_n=m}} V(\la)\T V(\la)^*}
{\bigoplus_{\substack{\la:\ k-1\ge \la_1,\ \la_n\ge 0\\ \la_1+\dots + \la_n=m-1}} V(\la)\T V(\la)^*},
\]
where the $\gl_n$-module $V(\la)$ is regarded as $\msl_n$-module
with highest weight $(\la_1-\la_n,\dots,\la_{n-1}-\la_n)$.

In this paper we solve the problems \eqref{dimeq} and \eqref{grch} for
$\g=\msp_{2m}$
and obtain certain results for orthogonal algebras.

\section{The Lie algebra $\msp_{4m}$ and the $C_2$-algebra $A(\msp_{2m};k)$}\label{sympgrading}
%\section{The graded character in type ${\tt C}_{m}$}\label{sympgrading}
%Throughout the paper (except for the last section)
%we restrict ourselves to the case $\g=\msp_{2m}$. So
%sometime we omit $\g$ in the notation of Zhu's and $C_2$-algebras and just write
%$A(k), A_{[2]}(k)$ and $A_{[2]}^m(k)$ instead of
%$A(\msp;k), A_{[2]}(\msp;k)$ and $A_{[2]}^m(\msp;k)$
%for the algebras respectively its homogeneous parts.

The idea of the following construction is to realize $A_{[2]}(\msp_{2m};k)$ as a representation of the much
larger algebra $\msp_{4m}$, and then use restriction algorithm arguments to prove a graded
character formula as well as the equality of the dimensions of the $C_2$-algebra and
Zhu's algebra.
\vskip 3pt
\noindent
\subsection{}
The enumeration of the fundamental weights is as in \cite{B}.
Let $\omega_1,\dots,\omega_m$ be the set of fundamental weights for the Lie algebra $\msp_{2m}$.
The highest root is $\theta=2\omega_1$, and for a dominant
integral $\la=\sum_{i=1}^m a_i\omega_i$ the condition
$(\la,\theta)\le k$ reads as $\sum_{i=1}^m a_i\le k$. Recall that for
any dominant weight $\la\in P^+(\msp_{2m})$ we have $V(\la)\simeq V(\la)^*$, so
Theorem \ref{zhualgebra} can be reformulated as
\begin{equation}\label{sptheorem}
A(\msp_{2m};k)\simeq
\bigoplus_{\substack{\la=\sum_{i=1}^m a_i\omega_i\\ \sum_{i=1}^m a_i\le k}} V(\la)\T V(\la).
\end{equation}

Let $\fp_{2m}\subset \msp_{4m}$ be the maximal parabolic Lie-subalgebra associated to the fundamental
weight $\om_{2m}$. We fix a Levi decomposition $\fp_{2m}=\fl\oplus \fn$
and $\fg=\fn^-\oplus\fp_{2m}$, where $\fl\simeq\mgl_{2m}$
as Lie algebra, $\fn\simeq S^2(\bc^{2m})$ and $\fn^-\simeq S^2(\bc^{2m})^*$ as $\fl=\mgl_{2m}$-module. The $\mgl_{2m}$-representations on $\fn$ and $\fn^-$ remain irreducible 
with respect to the action of the subalgebra $\msp_{2m}\subset \mgl_{2m}$, 
both representations are isomorphic to the adjoint representation of the symplectic Lie algebra.
Summarizing we have:
\begin{lem}\label{fnisomorphism}
As $\fl=\mgl_{2m}$-module we have isomorphisms: $\fn\simeq  S^2(\bc^{2m})$ and 
$\fn^-\simeq  S^2(\bc^{2m})^*$, as $\msp_{2m}$-module we have isomorphisms:
$\fn\simeq \fn^-\simeq \msp_{2m}$.
%As $\fl=\mgl_{2m}$ as well as $\msp_{2m}\subset \mgl_{2m}$-module we have isomorphisms:
%$\fn\simeq  S^2(\bc^{2m})\simeq \msp_{2m}$ and $\fn^-\simeq  S^2(\bc^{2m})^*\simeq \msp_{2m}$.
\end{lem}
In the following we always 
assume that for $\ell\in\bn$ the symplectic group  $Sp_{2\ell}$ is defined to be the
group leaving invariant the skew symmetric form on $\bc^{2\ell}$ defined by the $2\ell\times 2\ell$-matrix:
$$
J=\left(\begin{array}{ccccc}0 & 0 & 0 & 0 & 1 \\0 & 0 & 0 & 1 & 0 \\0 & 0 & .\cdot\,{}^\cdot  & 0 & 0 \\0 & -1 & 0 & 0 & 0 \\-1 & 0 & 0 & 0 & 0\end{array}\right).
$$
For a $m\times m$ matrix $A$ let $A^{nt}$ be the transpose of a matrix with respect to the diagonal given by
${i+j}=2m+1$, i.e. for $A=(a_{i,j})$ the matrix
$A^{nt}=( a^{nt}_{i,j})$ is given by $a^{nt}_{i,j}=a_{2m+1-j,2m+1-i}$. The Lie algebra
of the symplectic group $Sp_{2m}$ can then be described as the following set of matrices:
$$
\msp_{2m}=\left\{\left(\begin{array}{cc}
A & B  \\ C & -A^{nt} \end{array}\right)\big\vert
A,B,C\in M_m, \,B=B^{nt},\, C=C^{nt}
\right\}
$$
with maximal torus $\ft=\diag(t_1,\ldots,t_m,-t_m\ldots,-t_1)$ and Borel subalgebra the upper
triangular matrices of the form above.

The Lie algebra of the symplectic group $Sp_{2m}\subset GL_{2m}$ embedded in the Levi subgroup
$GL_{2m}\subset Sp_{4m}$ can be seen as the set of matrices of the following form:
\begin{equation}\label{spembedd}
\msp_{2m}=\left\{\left(\begin{array}{cccc}
A & B & 0 & 0 \\ C & -A^{nt} & 0 & 0 \\ 0 & 0 & A & -B \\ 0 & 0 & -C & -A^{nt}
\end{array}\right)\big\vert
\begin{array}{c}
A,B,C\in M_m, \\ B=B^{nt}\\ C=C^{nt}
\end{array}
\right\}
\subset \msp_{4m}.
\end{equation}
There is also a maximal reductive sub-Lie-algebra of type ${\tt C}_{m}+ {\tt C}_{m}$
sitting inside $\mathfrak{sp}_{4m}$ in the following way:
$$
\msp_{2m}\oplus \msp_{2m}=\left\{
\left(
\begin{array}{cccc}
K & 0 & 0 & L \\ 0 & X & Y & 0 \\ 0 & Z & -X^{nt} & 0 \\ M & 0 & 0 & -K^{nt}
\end{array}
\right)
\big\vert
\begin{array}{c}
K,L,M,X,Y,Z \in M_m \\ L=L^{nt},M=M^{nt} \\ Y=Y^{nt},Z=Z^{nt}
\end{array}
\right\}
%\subset \msp_{4m},
$$
Let $\eins_m$ denote the $m\times m$ identity matrix and let $J'$ be the $2m\times 2m$-matrix of the form
$$
J'=\left(\begin{array}{cc} 0 & \eins_m \\ -\eins_m & 0 \end{array}\right).
$$
Inside the Lie subalgebra $\mathfrak{sp}_{2m}\oplus \mathfrak{sp}_{2m}$ we have a diagonally
embedded symplectic Lie algebra $\Delta(\mathfrak{sp}_{2m})$, where:
\begin{equation}
\label{deltaspembedd}
\Delta:\mathfrak{sp}_{2m}\hookrightarrow \mathfrak{sp}_{2m}\oplus \mathfrak{sp}_{2m},
\quad Q\mapsto (Q,J'Q{J'}^{-1})
\end{equation}
which has the same maximal torus as the embedded symplectic Lie subalgebra $\mathfrak{sp}_{2m}$ described
in (\ref{spembedd}). As a consequence we see:

\begin{lem}\label{charactersame}
For a $\mathfrak{sp}_{4m}$-representation $\VVV(\la)$ let 
$$
\res^{\mathfrak{sp}_{4m}}_{\mathfrak{sp}_{2m}}\VVV(\la)
$$
be the $\mathfrak{sp}_{2m}$-module obtained via the embedding in (\ref{spembedd}) and let
$$
\res^{\mathfrak{sp}_{4m}}_{\Delta(\mathfrak{sp}_{2m})}\VVV(\la)
$$
be the $\mathfrak{sp}_{2m}$-module obtained via the embedding in (\ref{deltaspembedd}).
Then  
$$
\res^{\mathfrak{sp}_{4m}}_{\mathfrak{sp}_{2m}}\VVV(\la)\simeq
\res^{\mathfrak{sp}_{4m}}_{\Delta(\mathfrak{sp}_{2m})}\VVV(\la).
$$
\end{lem}

\subsection{The Lie algebra $\msp_{4m}$ and $A_{[2]}(\msp_{2m};k)$.}
We fix the standard maximal torus (diagonal matrices) and Borel subalgebra (upper triangular matrices)
for $\mgl_{2m}$. Then the maximal torus and the Borel subalgebra of
$\msp_{2m}$ and $\mgl_{2m}$ are contained in each other. Let $\theta$ be the maximal
root in the root system of $\msp_{2m}$. We identify the module $\msp_{2m}$ again with $\fn^-$ (Lemma~\ref{fnisomorphism}).
Fix a highest root vector $x_\ta$, then
$x_\ta$ is a highest weight vector for the adjoint action of $\msp_{2m}$ as well as
for the irreducible action of $\mgl_{2m}$ on the same representation space.
It follows that
\begin{equation}
\label{rootsandweightsp}
x_\ta^{k+1}\in S^\bullet(\msp_{2m})=S^\bullet(S^2(\bc^{2m}))^*=S^\bullet(\fn^-)
\end{equation}
is a highest weight vector of weight $2(k+1)\omega_1$ for the action of the symplectic group
$Sp_{2m}$ and of weight $-2(k+1)\epsilon_{2m}$ for the general linear group $GL_{2m}$. One checks easily
the following connection between $x_\ta$ and the root vectors for the Lie algebra $\msp_{4m}$:
\begin{lem}\label{rootequality}
Let $X_{-\alpha_{2m}}\in\fn^-\subset\msp_{4m}$ be a root 
vector for the negative of the simple root $\alpha_{2m}$ in the root system
for the Lie algebra $\msp_{4m}$. With respect to the embedding in (\ref{spembedd}), $X_{-\alpha_{2m}}$
is a weight vector for the Lie algebra $\msp_{2m}$ of weight $\theta$ and can hence be identified with 
$x_\theta$. 
%In particular,
%$X_{-\alpha_{2m}}$ and $x_\ta$ (see (\ref{rootsandweightsp})) are conjugate
%(up to some non-zero scalar multiple) with respect to action of the Weyl group of
%$Sp_{2m}$ on $\fn$.
\end{lem}
To distinguish between the highest weight representations of the different groups, we
write $V(\la)$ for the $Sp_{2m}$-representations, $\VV(\la)$ for the $GL_{2m}$-representations and
$\VVV(\la)$ for the $Sp_{4m}$-representations of highest weight $\la$ (whenever this makes sense).

The irreducible $GL_{2m}$-module $U(\mgl_{2m})\circ x_\ta^{k+1}$
generated by $x_\ta^{k+1}$ is the module $\VV(2(k+1)\omega_1)^*$, and hence
remains irreducible when restricted to $Sp_{2m}$, i.e. we have the following
sequence of equalities of vector spaces:
$$
\begin{array}{rcl}
\label{samespace}
U(\mgl_{2m})\circ x_\ta^{k+1} &=& U(\msp_{2m})\circ x_\ta^{k+1}\\
&=& U(\msp_{2m})\circ X_{-\alpha_{2m}}^{k+1} \\
&=& U(\mgl_{2m})\circ X_{-\alpha_{2m}}^{k+1}.
\end{array}
$$
Let $\VVV(k\omega_{2m})$ be the irreducible $Sp_{4m}$-module of highest weight $k\omega_{2m}$.
The nilpotent radical $\fn$ of $\fp_{2m}$ is abelian (since $\om_{2m}$ is a cominuscule weight,
see for example \cite{FFL}). Recall the following isomorphism of $\fl$-modules
(see \cite{FFL}, Lemma 3.1):
$$
\VVV(k\omega_{2m})\otimes\bc_{-k\om_{2m}}\simeq S^\bullet(\fn^-)/\langle U(\fl)\circ x_\ta^{k+1}\rangle,
$$
where  $\langle \cdots \rangle$ denotes the ideal
generated by the corresponding subspace.
Combining this isomorphism with (\ref{samespace}), we get
as a consequence the following isomorphisms of $GL_{2m}$-modules
as well as $Sp_{2m}$-modules. In particular, the $C_2$-algebra $ A_{[2]}(\msp_{2m};k)$ inherits the structure
of a $GL_{2m}$-module:
\begin{lem}\label{representglsp}
\begin{eqnarray*}
\VVV(k\omega_{2m})\otimes\bc_{-k\om_{2m}}&\simeq&S^\bullet(\fn^-)/
\langle U(\mgl_{2m})\circ X_{-\alpha_{2m}}^{k+1}\rangle\\
&=&S^\bullet(\msp_{2m})/\langle U(\msp_{2m})\circ x_\ta^{k+1}\rangle\\
&\simeq& A_{[2]}(\msp_{2m};k).
\end{eqnarray*}
\end{lem}

\section{The graded character of the $C_2$-algebra for $\msp_{2m}$.}
%{\bf \ref{sympgrading}
%.2 The graded character of the $C_2$-algebra for .}
Let $P_{2m}\subset Sp_{4m}$ be the parabolic subgroup 
associated to the fundamental weight $\om_{2m}$.
By \cite{L1} it is known that the action of the Levi subgroup $L=GL_{2m}\subset P_{2m}$
on $Sp_{4m}/P_{2m}$ is spherical, i.e., a Borel subgroup of $L$ has a dense orbit in
$Sp_{4m}/P_{2m}$. As a consequence, the restriction of an irreducible $Sp_{4m}$-module
of highest weight $\ell\om_{2m}$, $\ell\in\bn$, to $L$ is multiplicity free. Let $\omega_0$
denote the trivial character.  The tables in \cite{L1} imply:
\begin{prop}\label{Gldecomp}
As $GL_{2m}$-module, the $C_2$-algebra $A_{[2]}(\msp_{2m};k)$ decomposes as:
\begin{eqnarray*}
A_{[2]}(\msp_{2m};k)&=&\bigoplus_{\substack{\la=2(a_1\om_1+\ldots+a_{2m-1}\om_{2m-1})\\ a_1+\ldots + a_{2m}\le k}}
\VV(\la)\otimes\bc_{2(a_{2m}-k)\om_{2m}}
\end{eqnarray*}
\end{prop}
\begin{proof} The lemma above and the decomposition formula in \cite{L1} imply:
\begin{eqnarray*}
(\res_{GL_{2m}}^{Sp_{4m}} \VVV(k\omega_{2m}))\otimes\bc_{-k\om_{2m}}\hskip -10pt&=&\hskip-10pt
\bigoplus_{\substack{\la=a_0\omega_0+a_1 2\om_1+\ldots+a_{2m}2\om_{2m}\\ a_0+a_1+\ldots + a_{2m}=k}}\VV(\la)\otimes\bc_{-2k\om_{2m}}\\
%&=&\bigoplus_{\substack{\la=a_1 2\om_1+\ldots+a_{2m}2\om_{2m}\\ a_1+\ldots + a_{2m}\le k}}
%V(\la)\\
\hskip -10pt&=&\hskip-10pt
\bigoplus_{\substack{\la=2(a_1\om_1+\ldots+a_{2m-1}\om_{2m-1})\\ a_1+\ldots + a_{2m}\le k}}
\VV(\la)\otimes\bc_{2(a_{2m}-k)\om_{2m}}.
\end{eqnarray*}
\end{proof}
The center $Z$ of $GL_{2m}$, $Z:=\{t.\eins_{2m}\mid t\in\bc^*\}$, 
acts on $\msp_{2m}=S^2(\bc^{2m})^*=\fn^-$ by $t^{-2}$ and
hence on $S^j(\msp_{2m})$ by $t^{-2j}$. For a Young diagram (or a partition) $\la$ 
the action of the center on a representation $\VV(\la)$ can be 
described as follows. The element $t.\eins_{2m}\in Z$ acts as a multiplication
by $t$ to the power given by the number of boxes in the diagram. 
This leads to the corollary:
\begin{cor}
The $j$-th graded component of the $C_2$-algebra
$A_{[2]}(\msp_{2m};k)=\bigoplus_{j\ge 0 }A^j_{[2]}(\msp_{2m};k)$
decomposes as $\gl_{2m}$-module as follows:
\begin{eqnarray*}
A^j_{[2]}(\msp_{2m};k) &=&
\sum_{\substack{\la=2(a_1\om_1+\ldots+a_{2m-1}\om_{2m-1})\\ a_1+\ldots + a_{2m}\le k\\
a_1+2a_2+3a_3+\ldots + (2m-1)a_{2m-1}+2ma_{2m}+j=2mk}}\hskip -20pt
 \VV(\la)\otimes\bc_{2(a_{2m}-k)\om_{2m}}
%\res^{\gl_{2m}}_{\msp_{2m}} \VV(\la)
\end{eqnarray*}
\end{cor}
\begin{proof}
We know that $t.\eins_{2m}\in Z$ acts on $\VV(\la)\otimes\bc_{2(a_{2m}-k)\om_{2m}}$ as
\[
t^{2(a_1+2a_2+\dots +(2m-1)a_{2m-1})}t^{2m(2a_{2m}-2k)}.
\]
Therefore the condition
\[
\VV(\la)\otimes\bc_{2(a_{2m}-k)\om_{2m}}\hk A^j_{[2]}(\msp_{2m};k)
\]
can be reformulated as
\[
a_1+2a_2+3a_3+\ldots + (2m-1)a_{2m-1}+2ma_{2m}+j=2mk,
\]
which proves the Corollary.
\end{proof}

The Corollary provides a first step for a formula for the decomposition  of
$A^j_{[2]}(\msp_{2m};k)$ as $\msp_{2m}$-module.
To deduce from the above a graded character formula, recall the following restriction rule
due to Littlewood \cite{Lw} respectively its generalization by Koike and Terada \cite{KT}.
In the $GL_{2m}$ respectively $Sp_{2m}$ setting it is often convenient
to use at the same time the language of partitions as well as the language of highest weights.
By abuse of notation we denote  for a dominant weight $\la=\sum_{i=1}^{2m} a_i\om_i$ by
$\la=(\la_1,\la_2,\ldots)$ also the associated partition given by the rule:
$$
\la_1=\sum _{i=1}^{2m} a_i, \quad \la_2=\sum _{i=2}^{2m} a_i,\quad\ldots\quad ,\quad
\la_{2m-1}=a_{2m-1}+a_{2m},\quad \la_{2m}=a_{2m}.
$$
For three such partitions $\la,\mu,\nu$ we denote by $N_{\mu,\nu}^\la$ the associated Littlewood-Richardson
coefficient, i.e. for a tensor product of $GL_{2m}$-representations,
$N_{\mu,\nu}^\la$ is the multiplicity of $\VV(\la)$ in the tensor product $\VV(\mu)\otimes \VV(\nu)$.

Note that for a partition $\la$ with more than $m$ parts a representation $V(\la)$ for the symplectic
group $Sp_{2m}$ is not defined. In \cite{KT}, section 2.4, one finds an algorithm (folding of Young diagrams)
which associates to a partition $\la$ a partition $\pi(\la)$ with less than or equal to $m$ parts,
and a sign which we denote by
$sign(\la)$. For such a partition Koike and Terada define as the character associated to this partition
$\chara V(\la):=\sign(\la)\chara V(\pi(\la))$.
If $\la$ has less than or equal to $m$ parts, then the sign is one and $\pi(\la)=\la$.

\begin{thm}\cite{KT}
Suppose that $\la$ is a partition having at most $2m$ positive parts.
Then
$$
\chara \big(\res^{GL_{2m}}_{Sp_{2m}}\VV(\la) \big)= \sum_\mu \sum_{(2\nu)^t} N_{\mu,(2\nu)^t}^\la \chara V(\mu)
$$
where the sum is over all nonnegative integer partitions $\mu$ and $\nu$, and
$(2\nu)^t$ denotes the transpose of the partition $2\nu$.
\end{thm}
As a consequence we get:
\begin{thm} \label{sp2mgrch}
The $j$-th graded component of the $C_2$-algebra
$A_{[2]}(\msp_{2m};k)=\bigoplus_{j\ge 0 }A^j_{[2]}(\msp_{2m};k)$
decomposes as $Sp_{2m}$-module as follows:
\begin{eqnarray*}
\chara A^j_{[2]}(\msp_{2m};k) &=&
\sum_{\substack{\la=2(a_1\om_1+\ldots+a_{2m-1}\om_{2m-1})\\ a_1+\ldots + a_{2m}\le k\\
a_1+2a_2+3a_3+\ldots + 2ma_{2m}+j=2mk}}
\chara(\res^{GL_{2m}}_{Sp_{2m}} \VV(\la))\\
&=&\sum_{\substack{\la=2(a_1\om_1+\ldots+a_{2m-1}\om_{2m-1})\\ a_1+\ldots + a_{2m}\le k\\
a_1+2a_2+3a_3+\ldots + 2ma_{2m}+j=2mk}}
\sum_\mu \sum_{(2\nu)^t} N_{\mu,(2\nu)^t}^\la \chara V(\mu)
%\bigoplus_{\mu\le \la}\,(\sum_{2\nu \le \la} N_{\mu,(2\nu)^t}^\la) V(\mu).
\end{eqnarray*}
\end{thm}
%\proof
%The center $Z$ of $GL_{2m}$, $Z:=\{t.\eins_{2m}\mid t\in\bc^*\}$, acts on $\msp_{2m}=S^2(\bc^2)$ by %$t^2$ and
%hence on $S^j(\msp_{2m})$ by $t^{2j}$. The action of the center can be read off the Young diagram of %the
%representation respectively the partition, it is the number of boxes of the diagram respectively the sum of %the parts.
%Hence we get the grading formula above as a consequence of Proposition~\ref{Gldecomp} and the %restriction formula.
\qed

\section{Comparing dimensions in type $C_m$}\label{sympcompar}
\begin{thm}\label{dimensioncompare}
The $C_2$-algebra $A_{[2]}(\msp_{2m};k)$ and Zhu's algebra $A(\msp_{2m};k)$ have the same
dimension.
\end{thm}
\proof
The proof is divided into several steps.

{\it Step 1:} We transform the problem into a restriction problem for another group.
By Lemma~\ref{representglsp}, we have an isomorphism of $Sp_{2m}$-modules:
$$
A_{[2]}(\msp_{2m};k)\simeq\res^{Sp_{4m}}_{Sp_{2m}}\VVV(k\om_{2m}),
$$
where the embedding $Sp_{2m}\hookrightarrow Sp_{4m}$ is given by (\ref{spembedd}).
By Lemma~\ref{charactersame} we know that
$$
\res^{\mathfrak{sp}_{4m}}_{\mathfrak{sp}_{2m}}\VVV(k\om_{2m})\simeq
\res^{\mathfrak{sp}_{4m}}_{\Delta(\mathfrak{sp}_{2m})}\VVV(k\om_{2m})
$$
where the second embedding was described in (\ref{deltaspembedd}).
This second embedding has the advantage that by restricting
the representation first to the maximal semisimple Lie subalgebra
$\msp_{2m}\oplus \msp_{2m}$, we get already a description
of $\VVV(k\om_{2m})$ as a sum of tensor products of $\msp_{2m}$-modules,
which will simplify the comparison with the decomposition of the Zhu algebra in Theorem~\ref{zhualgebra}.
\vskip 3pt
{\it Step 2:}  It remains to determine the $Sp_{2m}\times Sp_{2m}$-structure of the representation
$\res_{Sp_{2m}\times Sp_{2m}}^{Sp_{4m}} \VVV(k\omega_{2m})$. We use geometric methods:
let $P_{2m}\subset Sp_{4m}$ be the maximal parabolic subgroup
associated to the weight $\omega_{2m}$ and set 
$Y=Sp_{4m}/P_{2m}\subset{\mathbb{P}(\VVV(\omega_{2m}))}$.
We denote by $\Y\subset \VVV(\omega_{2m})$ the affine cone over the projective variety.

The group $Sp_{4m}$ acts on the affine variety $\Y$ and hence on its coordinate ring $\bc[\Y]$. As
$Sp_{4m}$-module, this ring is the direct sum:
$$
\bc[\Y]=\bigoplus_{\ell\ge 0} \VVV(\ell\omega_{2m}).
$$
The ring $\bc[\Y]$ is naturally graded with $\VVV(\ell\omega_{2m})$ as
$\ell$-th graded component.

Let $U\subset G=Sp_{2m}\times Sp_{2m}$ be the unipotent radical
of a Borel subgroup. In a representation of $G$ the $U$-fixed vectors are sums
of highest weight vectors. The ring $\bc[\Y]^U$ of $U$-invariant vectors
completely determines the structure of $\bc[\Y]$ as $G$-representation.

\begin{prop}\label{uinvariant}
The ring of $U$-invariant functions
$$
\bc[\Y]^U=\bigoplus_{k\ge 0} \VVV(k\omega_{2m})^U
$$
is a polynomial ring generated by its degree 1 elements of weight
$\omega_i\otimes\omega_i$ for $i=0,\ldots,m$, where $\om_0$ denotes the trivial weight.
\end{prop}

The proof will be given in Step 4. As an immediate consequence we get:
\begin{cor}\label{zerlegspsp}
$$
(\res_{Sp_{2m}\times Sp_{2m}}^{Sp_{4m}} \VVV(k\omega_{2m}))=
\bigoplus_{\substack{\la=\sum_{i=1}^m a_i\om_i\\ \sum a_i\le k}} V(\la)\otimes V(\la).
$$
\end{cor}
\vskip 3pt
{\it Step 3:} {\it Proof of Theorem~\ref{dimensioncompare}}.
The theorem follows now from Corollary~\ref{zerlegspsp} together with Lemma~\ref{representglsp},
Theorem~\ref{zhualgebra} and $(\ref{sptheorem})$.
\qed
\vskip 3pt
{\it Step 4:} It remains to prove Proposition~\ref{uinvariant}. A first step in this
direction is the following:
\begin{prop}
The action of $G=Sp_{2m}\times Sp_{2m}$ on $Y=Sp_{4m}/P_{2m}$ is spherical, i.e.
a Borel subgroup of $G$ has a dense orbit in $Y$.
\end{prop}
\proof
We use the local structure theorem \cite{BLV}, which states the following:

Let $G$ be a connected complex reductive algebraic group.
Suppose $Y$ is a normal $G$-variety and $y\in Y$ is such that
the stabilizer $G_y$ of $y$ is a parabolic subgroup of $G$, i.e. the
orbit $G.y$ is a projective variety. Let $Q$ be a parabolic subgroup
opposite to $G_y$. Denote by $Q^u$ the
unipotent radical of $Q$ and set $L:=G_y\cap Q$.
\begin{thm}\cite{BLV}
There exists a locally closed, affine subvariety $Z$ of $Y$ such that $Z$
contains $y$ and is stable
under the action of $L$, $Q^u.Z$ is open in $Y$, and the canonical
map $Q^u\times Z\rightarrow Q^u.Z$ is an isomorphism of varieties.
\end{thm}
In our situation we have $Y=Sp_{4m}/P_{2m}$, $y=\overline{1}$ is the class of the identity,
$G_y$ is $R=P_m\times P_m$, where $P_m\subset Sp_{2m}$ is the maximal parabolic
associated to the fundamental weight $\omega_m$.
Let $Q$ be the opposite parabolic subgroup, then $L=R\cap Q=GL_m\times GL_m$.

Denote by $Q^u$ the unipotent radical of $Q$ and let
$$
O=G.\overline{1} \simeq (Sp_{2m}\times Sp_{2m})/(P_m\times P_m)=G/R
$$
be the closed orbit in $Y$. If we apply the local structure theorem to this situation,
then we may assume that $Z$ is smooth since $Y$ is smooth. Moreover, the action of
$G$ on $Y$ is spherical if the action of $L$ on $Z$ is spherical. Consider
the normal bundle ${\mathcal N}$ of $O$ in $X$ with fibre
$N$ at the coset $\overline{1}$ of the identity in $G/R$.
Then $N$ is isomorphic (as $L$-module) to the tangent space $T_y Z$ of
the $L$-fixed point $y$. It follows now by Luna's slice theorem that the action
of $L$ on $Z$ is spherical if and only if the action on $N$ is spherical. Now
as representation of $L=GL_m\times GL_m$ we have that $N=\bc^m\otimes \bc^m$, which is
a spherical action.
\qed

Let $\Y$ be the affine cone over $Y$. Since the coordinate ring $\bc[\Y]$ of
the affine cone over $Y=Sp_{4m}/P_{2m}$ is a unique factorization
ring, it follows (see for example \cite{L1}, Lemma 1):
\begin{cor}
The ring of $U$-invariant functions:
$$
\bc[\Y]^U=\bigoplus_{\ell\ge 0} \VVV(\ell\omega_{2m})^U
$$
is a polynomial ring.
\end{cor}
{\it Step 5:} To finish the proof of Proposition~\ref{uinvariant} we have to show that the
generators of $\bc[\Y]^U$ are of degree 1 and have the desired weights.

To calculate the $Sp_{2m}\times Sp_{2m}$-character, recall that (see for example \cite{FH}, Theorem 17.5)
\begin{equation}\label{2m}
\chara \VVV(\om_{2m})=\chara \Lambda^{2m}(\bc^{2m}\oplus \bc^{2m}) - \chara\Lambda^{2m-2}(\bc^{2m} \oplus \bc^{2m}).
\end{equation}
More generally, for all $p=1,\dots, 2m$ one has
\begin{equation} \label{p}
\chara \VVV(\om_p)=\chara \Lambda^p(\bc^{2m}\oplus \bc^{2m}) - \chara\Lambda^{p-2}(\bc^{2m} \oplus \bc^{2m}).
\end{equation}
%%%%%%%%%%%%%%%%%%%%%%%%
%%%%%%%%%%%%%%%%%%%%%%%%
%%%%%%%%%%%%%%%%%%%%%%%%
%%%%%%%%%%%%%%%%%%%%%%%%
%%%%%%%%%%%%%%%%%%%%%%%%
Using \eqref{2m} and \eqref{p}, one easily verifies that, as $Sp_{2m}\times Sp_{2m}$-module, we have:
$$
\VVV(\omega_{2m})= \bc\oplus V(\omega_1)\otimes V(\omega_1)\oplus V(\omega_2)\otimes V(\omega_2)\oplus\ldots
\oplus V(\omega_m)\otimes V(\omega_m).
$$
Let $f_0,\ldots,f_m\in \VVV(\omega_{2m})$ be the
highest weight vectors for the $Sp_{2m}\times Sp_{2m}$-action, where $f_0$ is a
$Sp_{2m}\times Sp_{2m}$-invariant function and $f_i$ is of weight $\omega_i\otimes\omega_i$
for $i\ge 1$. Since $\bc[\Y]^U$ is a polynomial ring, the grading and weights imply
that these elements
are algebraically independent. In addition, if these elements do not generate the ring, then
necessarily the Krull dimension $\dim \bc[\Y]^U>m+1$.  So Proposition~\ref{uinvariant}
is a consequence of the following lemma:

\begin{lem}\label{dimensiospuinvariant}
The Krull dimension $\dim \bc[\Y]^U=m+1$.
\end{lem}
\proof
Since $\bc[\Y]$ is a UFD, $U$ is connected and has no non-trivial characters,
$\bc(\Y)^U$ is the quotient field of $\bc[\Y]^U$. By Rosenlicht's theorem \cite{R},
generic orbits of an arbitrary action of a linear algebraic group on an irreducible
algebraic variety are separated by rational invariants, which implies in our case
that $\trdg\bc(\Y)^U=\dim \Y-\dim (\hbox{\rm generic orbit})$.

The maximal unipotent subgroup of $GL_m\times GL_m$ acts freely on an open
subset of $\bc^m\otimes\bc^m$, so the local structure theorem (Step 4) shows that $U$
operates freely on an open subset of $Y$ and hence does so on $\Y$. Since the codimension
of a generic $U$-orbit in $\Y$ is $m+1$, this finishes the proof of the lemma and hence
of Proposition~\ref{uinvariant}.
\qed

\section{The orthogonal case}
In this section we consider the case of orthogonal algebras $\mso_{2m}$
and $\mso_{2m+1}$. Our goal is to answer questions \eqref{dimeq} and
\eqref{grch} from Section \ref{setup}. Unfortunately, we are not
able at the moment to answer these questions completely.
The reason is that the $C_2$-algebra
$A_{[2]}(\mso_n;k)$ is bigger than the representation $V(k\omega_n)$
of $\mso_{2n}$. Therefore we can control only a certain quotient of the
$C_2$-algebra. The details are given below.

%In the following section we consider the even orthogonal case.
%{\it We begin as in the symplectic case, but the arguments need to
%be adapted. Actually, at the moment I don't know exactly yet how to do it.
%But let's first see how far we get. I don't remember whether this is Bourbaki
%convention, but I use for the $D_n$ case
%$\om_{m-1}=\frac{1}{2}(\epsilon_1+\ldots+\epsilon_{n-1}-\epsilon_n)$
%and
%$\om_{m}=\frac{1}{2}(\epsilon_1+\ldots+\epsilon_{n-1}+\epsilon_n)$. But this is not really important and easy
%to change if necessary. It is important to remember that for $D_{2m}$ all representations are self dual.}

%As before, most of the time we omit $\g$ in the notation of Zhu's and $C_2$-algebras.
%The idea of the following construction is to realize {\bf a quotient of}
%$A_{[2]}(k)$ as a representation of a much
%larger group: $Spin_{4m}$, and then use restriction algorithm arguments.
%\vskip 3pt
\subsection{Even orthogonal case: the setup}
For the enumeration of fundamental weights we use the same notation as in \cite{B}.
Let $\omega_1,\dots,\omega_m$ be the set of fundamental weights of $\mso_{2m}$.
The highest root is $\theta=\omega_2$, and for
$\la=\sum_{i=1}^m a_i\omega_i$ the condition
$(\la,\theta)\le k$ can be reformulated as
\[
a_1+ \sum_{i=2}^{m-2} 2a_i + a_{m-1} + a_m\le k.
\]
We also note that for even $m$ all representations are self-dual
and for odd $m$ spin representations are dual to each other
$V(\omega_{m-1})^*\simeq V(\omega_m)$.
Thus, Theorem \ref{zhualgebra} for $\g=\mso_{2m}$ can be reformulated as
\[
A(\mso_{2m};k)\simeq
\bigoplus_{\substack{\la=\sum_{i=1}^m a_i\omega_i\\
a_1+ \sum_{i=2}^{m-2} 2a_i + a_{m-1} + a_m\le k}} V_\la\T V_\la^*,
\]
where $V(\la)^*=V(\la)$ for even $m$ and
\[
V(\sum_{i=1}^m a_i\omega_i)^*=V(\sum_{i=1}^{m-2} a_i\omega_i +
a_{m-1}\omega_m + a_m\omega_{m-1})
\]
for odd $m$.

Let $\fp_{2m}\subset \mso_{4m}$ be the maximal parabolic Lie-subalgebra associated to the fundamental
weight $\om_{2m}$. We fix a Levi decomposition $\fp_{2m}=\fl\oplus \fn$ and $\fg=\fp_{2m}\oplus \fn^-$, where $\fl\simeq\mgl_{2m}$ as Lie algebra and $\fn\simeq \Lambda^2(\bc^{2m})$ respectively
$\fn^-\simeq \Lambda^2(\bc^{2m})^*$ as $\fl=\mgl_{2m}$-module. The restriction
of the $\mgl_{2m}$-representation on $\fn$ (respectively $\fn^-$) to the subalgebra 
$\mso_{2m}\subset \mgl_{2m}$
remains irreducible, it is the adjoint representation of the orthogonal Lie algebra. Summarizing we have:
\begin{lem}\label{fnisomorphismso}
As $\mgl_{2m}$-module we have isomorphisms: $\fn\simeq  \Lambda^2(\bc^{2m})$,
$\fn^-\simeq  \Lambda^2(\bc^{2m})^*$, and as $\mso_{2m}$-module we have isomorphisms: $\fn\simeq \fn^- \simeq \mso_{2m}$.
%$\fn\simeq  \Lambda^2(\bc^{2m})\simeq \mso_{2m}$.
%As $\fl=\mgl_{2m}$ as well as $\mso_{2m}\subset \mgl_{2m}$-module we have isomorphisms:
%$\fn\simeq  \Lambda^2(\bc^{2m})\simeq \mso_{2m}$.
\end{lem}
In the following we always assume that for $\ell\in\bn$ the orthogonal group  
$SO_{2\ell}$ is defined to be the
group leaving invariant the symmetric form on $\bc^{2\ell}$ defined by the 
$2\ell\times 2\ell$-matrix:
$$
J=\left(\begin{array}{ccccc}0 & 0 & 0 & 0 & 1 \\0 & 0 & 0 & 1 & 0 \\0 & 0 & .\cdot\,{}^\cdot  & 0 & 0 \\0 & 1 & 0 & 0 & 0 \\ 1 & 0 & 0 & 0 & 0\end{array}\right).
$$
For a $m\times m$ matrix $A$ let $A^{nt}$ be the transpose as in section~\ref{sympgrading}.
%of a matrix with respect to the diagonal given by
%${i+j}=2m+1$, i.e. for $A=(a_{i,j})$ the matrix
%$A^{nt}=( a^{nt}_{i,j})$ is given by $a^{nt}_{i,j}=a_{2m+1-j,2m+1-i}$. 
The Lie algebra of the orthogonal group $SO_{2m}$ can then be described as the 
following set of matrices:
$$
\mso_{2m}=\left\{\left(\begin{array}{cc}
A & B  \\ C & -A^{nt} \end{array}\right)\big\vert
A,B,C\in M_m, \,B=-B^{nt},\, C=-C^{nt}
\right\}
$$
with maximal torus $\ft=\diag(t_1,\ldots,t_m,-t_m\ldots,-t_1)$ and Borel subalgebra the upper
triangular matrices of the form above.

The Lie algebra of the orthogonal group $SO_{2m}\subset GL_{2m}$ embedded in the Levi subgroup
$GL_{2m}\subset SO_{4m}$ can be seen as the set of matrices of the following form:
\begin{equation}\label{soembedd}
\mso_{2m}=\left\{\left(\begin{array}{cccc}
A & B & 0 & 0 \\ C & -A^{nt} & 0 & 0 \\ 0 & 0 & A & B \\ 0 & 0 & C & -A^{nt}
\end{array}\right)\big\vert
\begin{array}{c}
A,B,C\in M_m, \\ B=-B^{nt}\\ C=-C^{nt}
\end{array}
\right\}
\subset \mso_{4m}.
\end{equation}
There is also a maximal reductive sub-Lie-algebra of type ${\tt D}_{m}+ {\tt D}_{m}$
sitting inside $\mathfrak{so}_{4m}$ in the following way:
$$
\mso_{2m}\oplus \mso_{2m}=\left\{
\left(
\begin{array}{cccc}
K & 0 & 0 & L \\ 0 & X & Y & 0 \\ 0 & Z & -X^{nt} & 0 \\ M & 0 & 0 & -K^{nt}
\end{array}
\right)
\big\vert
\begin{array}{c}
K,L,M,X,Y,Z \in M_m \\ L=-L^{nt},M=-M^{nt} \\ Y=-Y^{nt},Z=-Z^{nt}
\end{array}
\right\}
%\subset \mso_{4m},
$$
Let $\eins_m$ denote the $m\times m$ identity matrix and let $J'$ be the $2m\times 2m$-matrix of the form
$$
J'=\left(\begin{array}{cc} 0 & \eins_m \\ \eins_m & 0 \end{array}\right).
$$
Inside the Lie subalgebra $\mathfrak{so}_{2m}\oplus \mathfrak{so}_{2m}$ we have a diagonally
embedded orthogonal Lie algebra $\Delta(\mathfrak{so}_{2m})$, where:
\begin{equation}
\label{deltasoembedd}
\Delta:\mathfrak{so}_{2m}\hookrightarrow \mathfrak{so}_{2m}\oplus \mathfrak{so}_{2m},
\quad Q\mapsto (Q,J'Q{J'}^{-1})
\end{equation}
which has the same maximal torus as the embedded orthogonal Lie subalgebra $\mathfrak{so}_{2m}$ described
in (\ref{soembedd}). As a consequence we see:
\begin{lem}\label{charactersameso}
For a $\mathfrak{so}_{4m}$-representation $\VVV(\la)$ let $\res^{\mathfrak{so}_{4m}}_{\mathfrak{so}_{2m}}\VVV(\la)$
be the $\mathfrak{so}_{2m}$-representation obtained via the embedding in (\ref{soembedd})
and denote by
$$
\res^{\mathfrak{so}_{4m}}_{\Delta(\mathfrak{so}_{2m})}\VVV(\la)
$$
the $\mathfrak{so}_{2m}$-representation obtained via the embedding in (\ref{deltasoembedd}).

Then  $\res^{\mathfrak{so}_{4m}}_{\mathfrak{so}_{2m}}\VVV(\la)\simeq
\res^{\mathfrak{so}_{4m}}_{\Delta(\mathfrak{so}_{2m})}\VVV(\la)$.
\end{lem}

\subsection{The Lie algebra $\mso_{4m}$ and a quotient of $A_{[2]}(\mso_{2m};k)$.}
We fix the standard maximal torus (diagonal matrices) and Borel subalgebra (upper triangular matrices)
for $\mgl_{2m}$, then the maximal torus and the Borel subalgebra of
$\mso_{2m}$ and $\mgl_{2m}$ are contained in each other. Let $\theta$ be the maximal
root in the root system of $\mso_{2m}$, we identify the module $\mso_{2m}$ again with $\fn^-$ (Lemma~\ref{fnisomorphismso}).
Fix a highest root vector $x_\ta$, then
$x_\ta$ is a highest weight vector for the adjoint action of $\mso_{2m}$ as well as
for the irreducible action of $\mgl_{2m}$ on the same representation space.
It follows that
\begin{equation}
\label{rootsandweightsso}
x_\ta^{k+1}\in S^\bullet(\mso_{2m})=S^\bullet(\Lambda^2(\bc^{2m}))^*=S^\bullet(\fn)
\end{equation}
is a highest weight vector of weight $(k+1)\omega_2$ for the action of the orthogonal group
$SO_{2m}$ and of weight $-(k+1)(\epsilon_{2m-1}+\epsilon_{2m})$
for the general linear group $GL_{2m}$. One checks easily
the following connection between $x_\ta$ and the root vectors for the Lie algebra $\mso_{4m}$:

\begin{lem}\label{rootequalityso}
Let $X_{-\alpha_{2m}}\in\fn\subset\mso_{4m}$ be a root vector for the negative of 
the simple root $\alpha_{2m}$ of the root system
of $\mso_{4m}$. With respect to the embedding in (\ref{soembedd}), $X_{-\alpha_{2m}}$
is a weight vector for the Lie algebra $\mso_{2m}$ of weight $\theta$. 
%In particular, $X_{\alpha_{2m}}$ and $x_\ta$ (see 
%(\ref{rootsandweightsso})) are conjugate (up to some non-zero scalar multiple)
%with respect to action of the Weyl group of $SO_{2m}$ on $\fn$.
\end{lem}

To distinguish between the highest weight representations of the different groups, we
write $V(\la)$ for the $\mso_{2m}$-representations, $\VV(\la)$ for the $\mgl_{2m}$-representations and
$\VVV(\la)$ for the $\mso_{4m}$-representations of highest weight $\la$ (whenever this makes sense).

The irreducible $\mgl_{2m}$-module $U(\mgl_{2m})\circ x_\ta^{k+1}$
generated by $x_\ta^{k+1}$ is of weight %$-2(k+1)\epsilon_{2m}$
$-(k+1)(\epsilon_{2m-1}+\epsilon_{2m})$, and hence {does not}
remain irreducible when restricted to $\mso_{2m}$, i.e. we have the following
sequence of inclusions of vector spaces:
\begin{equation}
\label{sosamespace}
U(\mgl_{2m})\hskip -2pt\circ \hskip -2pt X_{-\alpha_{2m}}^{k+1}
\hskip -2pt=\hskip -2pt U(\mgl_{2m})\circ x_\ta^{k+1}\hskip -2pt \supset
\hskip -2pt U(\mso_{2m})\hskip -2pt \circ\hskip -2pt x_\ta^{k+1}
\hskip -2pt =\hskip -2pt U(\mso_{2m})\hskip -2pt\circ\hskip -2pt X_{-\alpha_{2m}}^{k+1}.
\end{equation}
Let $\VVV(k\omega_{2m})$ be the irreducible $Spin_{4m}$-module of highest weight $k\omega_{2m}$.
The nilpotent radical $\fn$ of $\fp_{2m}$ is abelian (since $\om_{2m}$ is a cominuscule weight,
see for example \cite{FFL}). Recall the following isomorphism of $\fl$-modules
(see \cite{FFL}, Lemma 3.1):
$$
\VVV(k\omega_{2m})\otimes\bc_{-k\om_{2m}}\simeq
S^\bullet(\fn^-)/\langle U(\fl)\circ x_\ta^{k+1}\rangle,
$$
where  $\langle \cdots \rangle$ denotes the ideal
generated by the corresponding subspace.
Combining this isomorphism with (\ref{sosamespace}), we get
as a consequence the following morphisms of $\mgl_{2m}$-modules
respectively $\mso_{2m}$-modules. In particular, we obtain a quotient of the
$C_2$-algebra $A_{[2]}(\mso_{2m};k)$
as a $\mso_{4m}$-representation:
\begin{lem}\label{representglso}
\begin{eqnarray*}
 A_{[2]}(\mso_{2m};k)&=&S^\bullet(\mso_{2m})/\langle U(\mso_{2m})\circ x_\ta^{k+1}\rangle\\
 &\rightarrow & S^\bullet(\Lambda^2\bc^{2m})^*/\langle U(\mgl_{2m})\circ x_\ta^{k+1}\rangle\\
 &=&S^\bullet(\fn^-)/\langle U(\mgl_{2m})\circ X_{\alpha_{2m}}^{k+1}\rangle\\
&\simeq&\VVV(k\omega_{2m})\otimes\bc_{-k\om_{2m}}.
\end{eqnarray*}
\end{lem}

\begin{rem}\label{dualrem}
Of course, one can use the representation $\VVV(k\omega_{2m-1})$ instead
of $\VVV(k\omega_{2m})$. It turns out that
in order to compare Zhu's algebra and the $C_2$-algebra of type $D_m$ one has to use
the first representation for odd $m$ and the second one for even $m$.
We work out in details the even $m$ case. The odd case is considered in
subsection \ref{dual}.
\end{rem}

\subsection{The $\mso_{2m}\times \mso_{2m}$-decomposition of $\VVV(k\omega_{2m})$}
%\subsubsection{The $GL_{2m}$-decomposition.}
%The action of the Levi subgroup $GL_{2m}\subset P_{2m}\subset Spin_{4m}$
%on $Spin_{2m}/P_{2m}$ is spherical \cite{L1}. It follows from the results in
%\cite{L1} that
%\begin{prop}
%The $\mso_{4m}$-module $\VVV(k\omega_{2m})$ decomposes with respect to Levi Lie
%subalgebra $\mgl_{2m}$ (up to a twist by a $GL$-character) into the direct sum:
%$$
%\res^{Spin_{4m}}_{GL_{2m}} \VVV(k\omega_{2m})=\bigoplus_{\substack{i=0 \\ \sum a_i=k}}^m \VV(a_i\om_{2i})
%$$
%end{prop}
%%%%%%%%%%%%%%%%%%%%%%%%%%%%%%%%%%%%%%%%%%%%%%%%%%%%%%%%%%%%
%%%%%%%%%%%%%%%%%%%%%%%%%%%%%%%%%%%%%%%%%%%%%%%%%%%%%%%%%%%%
In the following we investigate the decomposition of $\VVV(k\omega_{2m})$ as
$\mso_{2m}$-module. By Lemma~\ref{charactersameso} it suffices to describe
the $\mso_{2m}\times \mso_{2m}$-module structure.

We proceed as in the symplectic case and use geometric methods:
let $P_{2m}\subset Spin_{4m}$ be the maximal parabolic subgroup
associated to the weight $\omega_{2m}$ and set $Y=Spin_{4m}/P_{2m}\subset{\mathbb{P}(\VVV(k\omega_{2m}))}$.
We denote by $\Y\subset \VVV(k\omega_{2m})$ the affine cone over the projective variety.

The group $Spin_{4m}$ acts on the affine variety $\Y$ and hence on its coordinate ring $\bc[\Y]$. As
$Spin_{4m}$-module, this ring is the direct sum:
$$
\bc[\Y]=\bigoplus_{\ell\ge 0} \VVV(\ell\omega_{2m}).
$$
The ring $\bc[\Y]$ is naturally graded with $\VVV(\ell\omega_{2m})$ as
$\ell$-th graded component.
Let $U\subset G=Spin_{2m}\times Spin_{2m}$ be the unipotent radical
of a Borel subgroup. 
%In a representation of $G$ the $U$-fixed vectors are sums
%of highest weight vectors. 
%The ring $\bc[\Y]^U$ of $U$-invariant vectors
%completely determines the structure of $\bc[\Y]$ as $G$-representation.

\begin{prop}\label{uinvariantso}
The ring of $U$-invariant functions
$$
\bc[\Y]^U=\bigoplus_{k\ge 0} \VVV(k\omega_{2m})^U
$$
is a polynomial ring generated by its degree 1 elements of weight
$\omega_m\otimes\omega_m$ and $\omega_{m-1}\otimes\omega_{m-1}$
and the degree 2 elements $\om_i\otimes \om_i$ for $i=0,\ldots,m-2$,
where $\om_0$ denotes the trivial weight.
\end{prop}

As an immediate consequence we get:
\begin{cor}\label{zerlegspso}
$$
\res_{Spin_{2m}\times Spin_{2m}}^{Spin_{4m}} \VVV(k\omega_{2m})=
\bigoplus_{\substack{\la=\sum_{i=1}^m a_i\om_i\\ (\sum_{i=1}^{m-2} 2a_i)+a_{m-1}+a_m\le k \\
k-a_{m-1}-a_m\equiv 0\pmod 2}} V(\la)\otimes V(\la).
$$
\end{cor}
%%%%%%%%%%%%%%%%%%%%%%%%%%%%%%%%%%%%%%%%
%%%%%%%%%%%%%%%%%%%%%%%%%%%%%%%%%%%%%%%%
It remains to prove Proposition~\ref{uinvariantso}. A first step in this
direction is the following:
\begin{prop}
The action of $G=Spin_{2m}\times Spin_{2m}$ on $Y=Spin_{4m}/P_{2m}$ is spherical, i.e.
a Borel subgroup of $G$ has a dense orbit in $Y$.
\end{prop}
\begin{proof}
As before we use the local structure theorem \cite{BLV}:
In our situation we have $Y=Spin_{4m}/P_{m}$, $y=\overline{1}$ is the class of the identity,
$G_y$ is $R=P_m\times P_m$, where $P_m\subset Spin_{2m}$ is the maximal parabolic
associated to the fundamental weight $\omega_m$.
Let $Q$ be the opposite parabolic subgroup, then $\fl=\mgl_m\times \mgl_m$.

Denote by $Q^u$ the unipotent radical of $Q$ and let
$$
O=G.\overline{1} \simeq (Spin_{2m}\times Spin_{2m})/(P_m\times P_m)=G/R
$$
be the closed orbit in $Y$. If we apply the local structure theorem to this situation,
then we may assume that $Z$ is smooth since $Y$ is smooth. Moreover, the action of
$G$ on $Y$ is spherical if the action of $L$ on $Z$ is spherical. Consider
the normal bundle ${\mathcal N}$ of $O$ in $X$ with fibre
$N$ at the coset $\overline{1}$ of the identity in $G/R$.
Then $N$ is isomorphic (as $L$-module) to the tangent space $T_y Z$ of
the $L$-fixed point $y$. It follows now by Luna's slice theorem that the action
of $L$ on $Z$ is spherical if and only if the action on $N$ is spherical. Now
as representation for the Lie algebra $\fl=\mgl_m\oplus \mgl_m$ we get the action
on $N=\bc^m\otimes \bc^m$, which is a spherical action.
\end{proof}

Let $\Y$ be the affine cone over $Y$. Since the coordinate ring $\bc[\Y]$ of
the affine cone over $Y=Spin_{4m}/P_{2m}$ is a unique factorization
ring, it follows:
\begin{cor}
The ring of $U$-invariant functions:
$$\bc[\Y]^U=\bigoplus_{\ell\ge 0} \VVV(\ell\omega_{2m})^U$$
is a polynomial ring.
\end{cor}
\noindent
{\it Proof of Proposition~\ref{uinvariantso}.\/}
To finish the proof of Proposition~\ref{uinvariantso} we have to show that the
generators of $\bc[\Y]^U$ are of the desired degrees and weights.

To calculate the $\mso_{2m}\oplus \mso_{2m}$-character, recall that
(see for example \cite{FH}, Theorem 19.2) for the $Spin_{4m}$-modules we have:
$$
\Lambda^{2m}\bc^{4m}=\Lambda^{2m}(\bc^{2m}\oplus \bc^{2m})= \VVV(2\om_{2m})\oplus \VVV(2\om_{2m-1}).
$$
Using the decomposition $\bc^{4m}=(\bc^{2m}\oplus \bc^{2m})$,
with this description one verifies easily that, as $\mso_{2m}\oplus \mso_{2m}$-module, we have:
$$
\begin{array}{rcl}
\res_{\mso_{2m}\oplus \mso_{2m}}^{\mso_{4m}}
\VVV(2\omega_{2m})&=& \bc\oplus V(\omega_1)\otimes V(\omega_1)\oplus\ldots \oplus
V(\omega_{m-2})\otimes V(\omega_{m-2})\\
& &  \oplus  V(\omega_{m-1}+\omega_{m})\otimes V(\omega_{m-1}+\omega_{m})\\
& &\oplus V(2\omega_{m-1})\otimes V(2\omega_{m-1})\oplus V(2\omega_m)\otimes V(2\omega_m).
\end{array}
$$
For the fundamental representation one computes:
$$
\res_{\mso_{2m}\oplus \mso_{2m}}^{\mso_{4m}}\VVV(\omega_{2m})=
V(\omega_{m-1})\otimes V(\omega_{m-1})\oplus V(\omega_{m})\otimes V(\omega_{m})
$$

Let $f_0,\ldots,f_m\in \VVV(\omega_{2m})\oplus \VVV(2\omega_{2m})$ be
highest weight vectors for the $\mso_{2m}\oplus \mso_{2m}$-action, where $f_0$ is a
$\mso_{2m}\oplus \mso_{2m}$-invariant function of degree 2, $f_1,\ldots,f_{m-2}$ are of
weight $\omega_i\otimes\omega_i$ for $i=1,\ldots, m-2$ and of degree 2, and $f_{m-1},f_m$
are of degree 1 and of weight $\omega_{m-1}\otimes \omega_{m-1}$ respectively $\omega_{m}\otimes \omega_{m}$.

The collection of functions $f_0,\ldots,f_m,f_{m-1}^2, f_{m-1}f_m,f_m^2$
is a basis for the subspace of highest weight vectors in $\VVV(\omega_{2m})\oplus \VVV(2\omega_{2m})$.

Since $\bc[\Y]^U$ is a polynomial ring, the grading and the weights imply that the elements $f_0,\ldots,f_m$
are algebraically independent. In addition, if these elements do not generate the ring, then
necessarily $\dim \bc[\Y]^U>m+1$.  So Proposition~\ref{uinvariantso}
is a consequence of the following lemma, which is proved along the same lines
as Lemma~\ref{dimensiospuinvariant}:
\qed
\begin{lem}
$\dim \bc[\Y]^U=m+1$.
\end{lem}

\subsection{The dual realization}\label{dual}
Recall that for odd $m$ not all representations of $\mso_{2m}$
are self-dual. Therefore, Corollary \ref{zerlegspso} does not allow
to compare Zhu's algebras and $C_2$-algebras. In fact, to handle this problem, we
consider the weight $\omega_{2m-1}$ of $\mso_{4m}$ (instead of of $\omega_{2m}$).
This weight is also cominuscule and thus everything works for $\omega_{2m-1}$
as well. Below we formulate the analogues of Proposition \ref{uinvariantso}
and Corollary \ref{zerlegspso}.

Fix an odd $m$. Let $P_{2m-1}\subset Spin_{4m}$ be the maximal parabolic subgroup
associated to the weight $\omega_{2m-1}$ and set $Y'=Spin_{4m}/P_{2m-1}\subset{\mathbb{P}(\VVV(k\omega_{2m-1}))}$.
We denote by $\Y'\subset \VVV(k\omega_{2m-1})$ the affine cone over the projective variety.

%The group $Spin_{4m}$ acts on the affine variety $\Y$ and hence on its coordinate ring $\bc[\Y]$. As
%$Spin_{4m}$-module, this ring is the direct sum:
%$$
%\bc[\Y]=\bigoplus_{\ell\in\bn} \VVV(\ell\omega_{2m}).
%$$
%The ring $\bc[\Y]$ is naturally graded with $\VVV(\ell\omega_{2m})$ as
%$\ell$-th graded component.

%Let $U\subset G=Spin_{2m}\times Spin_{2m}$ be the unipotent radical
%of a Borel subgroup, in a representation of $G$ the $U$-fixed vectors are sums
%of highest weight vectors. The ring $\bc[\Y]^U$ of $U$-invariant vectors
%completely determines the structure of $\bc[\Y]$ as $G$-representation.

\begin{prop}
The ring of $U$-invariant functions
$$
\bc[\Y']^U=\bigoplus_{k\ge 0} \VVV(k\omega_{2m-1})^U
$$
is a polynomial ring generated by its degree 1 elements of weight
$\omega_m\otimes\omega_{m-1}$ and $\omega_{m-1}\otimes\omega_m$
and the degree 2 elements $\om_i\otimes \om_i$ for $i=0,\ldots,m-2$,
where $\om_0$ denotes the trivial weight.
\end{prop}

As an immediate consequence we get:
\begin{cor}
$$
\res_{Spin_{2m}\times Spin_{2m}}^{Spin_{4m}} \VVV(k\omega_{2m-1})=
\bigoplus_{\substack{\la=\sum_{i=1}^m a_i\om_i\\ (\sum_{i=1}^{m-2} 2a_i)+a_{m-1}+a_m\le k \\
k-a_{m-1}-a_m\equiv 0\pmod 2}} V(\la)\otimes V(\la)^*.
$$
\end{cor}

\subsection{Comparison}
Summarizing, we obtain the following:
\begin{prop}
We have an isomorphism of $\mso_{2m}$ modules
\begin{equation}\label{quot}
S^\bullet(\mso_{2m})/\langle \U(\gl_{2m})\circ e_\theta^{k+1} \rangle\simeq
\bigoplus_{\substack{\la=\sum_{i=1}^m a_i\om_i\\
(\sum_{i=1}^{m-2} 2a_i)+a_{m-1}+a_m\le k \\
k-a_{m-1}-a_m\equiv 0\pmod 2}} V(\la)\otimes V(\la)^*,
\end{equation}
providing a surjection of $A_{[2]}(\mso_{2m};k)$ to the right hand side
of \eqref{quot}.
\end{prop}

\subsection{The odd orthogonal case.}
In this subsection we consider the case $\g=\mso_{2m+1}$. All the constructions
as above are valid in this case as well, so we only formulate the final result in the following
proposition.

\begin{prop}
There exists a surjection of $\mso_{2m+1}$ modules:
\[
A_{[2]}(\mso_{2m+1};k)\to \bigoplus_{\substack{\la=\sum_{i=1}^m a_i\om_i\\
(\sum_{i=1}^{m-1} 2a_i)+a_m\le k \\
k-a_m\equiv 0\pmod 2}} V(\la)\otimes V(\la)^*.
\]
\end{prop}

%Now comes what I know and where at the moment I don't know how to proceed to prove what we want.
%We want:
%$$
%A_{[2]}(k)=\bigoplus_{\substack{\la=\sum_{i=1}^m a_i\om_i\\ a_1+(\sum_{i=2}^{m-2} 2a_i)+a_{m-1}+a_m\le k}}
%V(\la)\otimes V(\la).
%$$
%and we have so far only the following part: There is an ideal $I\subset A_{[2]}(k)$ such that
%$$
%A_{[2]}(k)/I=
%\bigoplus_{\substack{\la=\sum_{i=1}^m a_i\om_i\\ (\sum_{i=1}^{m-2} 2a_i)+a_{m-1}+a_m\le k \\
%k-a_{m-1}-a_m\equiv 0\pmod 2}} V(\la)\otimes V(\la).
%$$
%%Everything would be nice if the trivial and the $\om_1\otimes\om_1$ representation would have shown
%up in $\VVV(\om_{2m})$ instead of only in $\VVV(2\om_{2m})$, but the reason for this is of course
%that in the construction via the abelian radical the ideal to describe the representation
%is bigger than the ideal for the definition of $A_{[2]}(k)$.

%So let's try to understand how much "bigger" the generating set for the ideal for the representation
%really is.

%\begin{lem}
%Denote by $\om_0$ the trivial weight.
%The irreducible $\mgl_{2m}$-module $U(\mgl_{2m})\circ x_\theta^{k+1}$
%decomposes as $\mso_{2m}$-module into
%$$
%U(\mgl_{2m})\circ x_\theta^{k+1}=
%\bigoplus_{\substack{\la=2a_1\om_1+a_2\om_2\\ 2a_1 +a_2\le k+1 \\ k+1-2a_1 -a_2\equiv 0 \bmod 2 }}
%V(\la)
%%$$
%The ring of $U$-invariants is generated by three elements, one of degree 1 and weight $\om_2$,
%and two of degree 2, of which one is an invariant and the other is of weight $2\om_1$.
%\end{lem}
%Proof is similar to the proofs before.

\section*{Acknowledgements}
We are grateful to Terry Gannon for useful correspondence.
The work of EF was partially supported
by the Russian President Grant MK-281.2009.1, the RFBR Grants 09-01-00058, 07-02-00799
and NSh-3472.2008.2, by Pierre Deligne fund
based on his 2004 Balzan prize in mathematics and by Alexander von
Humboldt Fellowship.
The work of Peter Littelmann was partially supported by the
priority program SPP 1377 of the German Science Foundation.

\end{document}